\documentclass{amsart}%
\usepackage{amsfonts}
\usepackage{amsmath}
\usepackage{amssymb}
\usepackage{graphicx}%
\newtheorem{theorem}{Theorem}
\theoremstyle{plain}

\newtheorem{conjecture}{Conjecture}
\newtheorem{corollary}{Corollary}

\newtheorem{lemma}{Lemma}

\newtheorem{proposition}{Proposition}

\numberwithin{equation}{section}
\begin{document}
\title[Rigidity Theorems]{Rigidity Theorems for Compact Manifolds with Boundary and Positive Ricci Curvature}
\author{Fengbo Hang}
\address{Courant Institute, 251 Mercer Street, New York, NY 10012}
\email{fengbo@cims.nyu.edu}
\author{Xiaodong Wang}
\address{Department of Mathematics, Michigan State University, East Lansing, MI 48864}
\email{xwang@math.msu.edu}

\begin{abstract}
We prove some boundary rigidity results for the hemisphere under a lower bound
for Ricci curvature. The main result can be viewed as the Ricci version of a
conjecture of Min-Oo.

\end{abstract}
\maketitle

\section{\bigskip Introduction}

The positive mass theorem, first proved by Schoen-Yau \cite{SY1, SY2} and
later by Witten \cite{W} using spinors, is one of the profound results in
differential geometry. In the recent work of Shi-Tam, it is used in a novel
way to yield beautiful results on the boundary effect on compact Riemannian
manifolds with nonnegative scalar curvature. The following theorem is only a
special case of their main result.

\begin{theorem}
\label{ball}Let $\left(  M^{n},g\right)  $ be a compact Riemannian manifold
with boundary and scalar curvature $R\geq0$. If the boundary is isometric to
$\mathbb{S}^{n-1}$ and has mean curvature $n-1$, then $\left(  M^{n},g\right)
$ is isometric to the unit ball $\overline{\mathbb{B}^{n}}\subset
\mathbb{R}^{n}$. (If $n>7$ we need to assume that $M$ is spin.)
\end{theorem}

We sketch the idea of the proof. We glue $M$ with $\mathbb{R}^{n}%
\backslash\mathbb{B}^{n}$ along the boundary $\mathbb{S}^{n-1}$ to obtain an
asymptotically flat manifold $N$ with nonnegative scalar curvature. Since it
is actually flat near infinity the positive mass theorem implies that $N$ is
isometric to $\mathbb{R}^{n}$ and hence $M$ is isometric to $\overline
{\mathbb{B}^{n}}$ (see \cite{M, ST} for details). There are similar rigidity
results for geodesic balls in the hyperbolic space assuming $R\geq-n\left(
n-1\right)  $ by applying the positive mass theorem for asymptotically
hyperbolic manifolds.

It is a natural question to consider the hemisphere. The following conjecture
was proposed by Min-Oo in 1995.

\begin{conjecture}
(Min-Oo) Let $\left(  M^{n},g\right)  $ be a compact Riemannian manifold with
boundary and scalar curvature $R\geq n\left(  n-1\right)  $. If the boundary
is isometric to $\mathbb{S}^{n-1}$ and totally geodesic, then $\left(
M^{n},g\right)  $ is isometric to the hemisphere $\mathbb{S}_{+}^{n}$.
\end{conjecture}

The proof of Theorem \ref{ball} does not seem to work any more: there is no
positive mass theorem providing a miraculous passage from the compact manifold
in question to a noncompact manifold. As it stands this conjecture seems quite
difficult. There have only been some partial results in \cite{HW} and some
recent progress in dimension three in \cite{E}.

Motivated by Min-Oo's conjecture, we consider the rigidity of compact
Riemannian manifolds with boundary and positive Ricci curvature. Here is our
first result.

\begin{theorem}
\label{main}Let $\left(  M^{n},g\right)  $ ($n\geq2$) be a compact Riemannian
manifold with nonempty boundary $\Sigma=\partial M$. Suppose

\begin{itemize}
\item \textrm{$Ric$}$\geq\left(  n-1\right)  g,$

\item $\left(  \Sigma,g|_{\Sigma}\right)  $ is isometric to the standard
sphere $\mathbb{S}^{n-1}\subset\mathbb{R}^{n}$,

\item $\Sigma$ is convex in $M$ in the sense that its second fundamental form
is nonnegative.
\end{itemize}

Then $\left(  M^{n},g\right)  $ is isometric to the hemisphere $\mathbb{S}%
_{+}^{n}=\{x\in\mathbb{R}^{n+1}:|x|=1,x_{n+1}\geq0\}\subset\mathbb{R}^{n+1}$.
\end{theorem}

Since there are different conventions for the second fundamental form and the
mean curvature in the literature, let us explain ours. Let $\nu$ be the outer
unit normal field of $\Sigma$ in $M$. For any $p\in\Sigma$, for any $X,Y\in
T_{p}\Sigma$ the second fundamental form is defined as%
\[
\Pi\left(  X,Y\right)  =\left\langle \nabla_{X}\nu,Y\right\rangle .
\]
The mean curvature is the trace of the second fundamental form.

Put in another way, the theorem says that for a compact manifold with
boundary, if we know that the boundary is $\mathbb{S}^{n-1}$(intrinsic
geometry on the boundary) and convex (some extrinsic geometry) then we
recognize the manifold as the hemisphere $\mathbb{S}_{+}^{n}$, provided
$\mathrm{Ric}\geq\left(  n-1\right)  g$. Theorem \ref{main} can be viewed as
the Ricci version of Min-Oo's conjecture. It is a strong evidence that
Min-Oo's conjecture should be true.

Shi-Tam \cite{ST} have also studied compact manifolds $\left(  M^{n},g\right)
$ whose boundaries isometrically embed in $\mathbb{R}^{n}$ as a convex
hypersurface. In our case we may consider compact Riemannian manifolds whose
boundaries isometrically embed as a hypersurface in $\mathbb{S}_{+}^{n}$. \ We
prove the following rigidity theorem in this more general case.

\begin{theorem}
\label{general}Let $\left(  M,g\right)  $ be a smooth compact Riemannian
manifold with boundary $\partial M=\Sigma$ and $\overline{\Omega}%
\subset\mathbb{S}_{+}^{n}$ is a compact domain with smooth boundary in the
open hemisphere. Suppose

\begin{itemize}
\item \textrm{$Ric$}$\geq\left(  n-1\right)  g,$

\item there is an isometric embedding $\iota:\left(  \Sigma,g_{\Sigma}\right)
\rightarrow\partial\overline{\Omega}$,

\item $\Pi\geq\Pi_{0}\circ\iota$, \ here $\Pi$ the second fundamental form of
$\Sigma$ in $M$ and $\Pi_{0}$ is the second fundamental form of $\partial
\overline{\Omega}$ in $\mathbb{S}_{+}^{n}.$
\end{itemize}

Then $\left(  M,g\right)  $ is isometric to $\left(  \overline{\Omega
},g_{\mathbb{S}_{+}^{n}}\right)  $.
\end{theorem}

In dimension 2 it turns out that Theorem \ref{main} is essentially equivalent
to a result of Toponogov on the length of simple closed geodesics on a
strictly convex surface. This connection is discussed in Section 2 in which we
also present a different proof working only in dimension 2. This proof may
have some independent interest. It is also interesting to compare this two
dimensional argument, which is partly geometric and partly analytic, with the
unified proof of purely analytic nature presented in Section 4.

To prove Theorem \ref{general}, we have to generalize the proof of Theorem
\ref{main} to the more general context where we allow the metric to be
Lipschitz along a hypersurface. For this purpose we first establish Reilly's
theorem on the first eigenvalue in this context in Section 5. The proof of
Theorem \ref{general} is then given in Section 6.

\textbf{Acknowledgement: }The research of F. Hang is supported by National
Science Foundation Grant DMS-0647010 and a Sloan Research Fellowship. The
research of X. Wang is supported by National Science Foundation Grant
DMS-0505645. We would like to thank Christina Sormani for valuable discussions.

\section{\bigskip The two dimensional case}

When $n=2$ we consider a compact surface $(M^{2},g)$ with boundary. The
boundary then consists of closed curves and there is no intrinsic geometry
except the lengths of these curves. The extrinsic geometry of the boundary is
given by the geodesic curvature. Therefore Theorem \ref{main} follows from the
following slightly stronger result.

\begin{theorem}
Let $(M^{2},g)$ be compact surface with boundary and the Gaussian curvature
$K\geq1.$ Suppose the geodesic curvature $k$ of the boundary $\gamma$
satisfies $k\geq c$ $\geq0$. Then $L(\gamma)\leq2\pi/\sqrt{1+c^{2}}$. Moreover
equality holds iff $(M,g)$ is isometric to a disc of radius $\cot^{-1}(c)$ in
$\mathbb{S}^{2}$.

\begin{proof}
By Gauss-Bonnet formula
\[
2\pi\chi\left(  M\right)  =\int_{M}Kd\sigma+\int_{\gamma}kds>0,
\]
where $\chi\left(  M\right)  $ is the Euler number of $M$. Therefore $M$ is
simply connected and in particular $\gamma$ has only one component. By the
Riemann mapping theorem, $(M,g)$ is conformally equivalent to the unit disc
$\overline{\mathbb{B}}\mathbb{=}\left\{  z\in\mathbb{C}:\left\vert
z\right\vert \leq1\right\}  $. Without loss of generality, we take $(M,g)$ to
be $(\overline{\mathbb{B}},g=e^{2u}|dz|^{2})$ with $u\in C^{\infty}\left(
\overline{\mathbb{B}},\mathbb{R}\right)  $. By our assumptions we have%

\[
\left\{
\begin{array}
[c]{c}%
-\Delta u\geq e^{2u}\text{ on }\overline{\mathbb{B}},\\
\frac{\partial u}{\partial r}+1\geq ce^{u}\text{ \ on }\mathbb{S}^{1}%
\end{array}
\right.
\]
Let $\underline{u}\in C^{\infty}\left(  \overline{\mathbb{B}},\mathbb{R}%
\right)  $ such that%
\[
\left\{
\begin{array}
[c]{c}%
-\Delta\underline{u}=0\text{ on }B,\\
\left.  \underline{u}\right\vert _{\mathbb{S}^{1}}=\left.  u\right\vert
_{\mathbb{S}^{1}}.
\end{array}
\right.
\]
Then $\underline{u}\leq u$ as $u$ is superharmonic. It follows from sub-sup
solution method (see, e.g., \cite[page 187-189]{SY}) that we may find a $v\in
C^{\infty}\left(  \overline{\mathbb{B}},\mathbb{R}\right)  $ with%
\[
\left\{
\begin{array}
[c]{c}%
-\Delta v=e^{2v}\text{ on }\overline{\mathbb{B}},\\
\underline{u}\leq v\leq u.
\end{array}
\right.
\]
Since $v\leq u$ and $v|_{\mathbb{S}^{1}}=u|_{\mathbb{S}^{1}}$ we have
$\frac{\partial v}{\partial\nu}\geq\frac{\partial u}{\partial\nu}$ and hence
$\frac{\partial v}{\partial\nu}+1\geq ce^{u}$ on $\mathbb{S}^{1}$, i.e. the
boundary circle has has geodesic curvature $\geq c$. As the metric $\left(
\overline{\mathbb{B}},e^{2v}|dz|^{2}\right)  $ has curvature $1$ and the
boundary circle is convex, it can be isometrically embedded as a domain in
$\mathbb{S}^{2}$, say $\Omega$. Denote $\sigma=\partial\Omega$ parametrized by
arclength. Notice $L\left(  \sigma\right)  =L\left(  \gamma\right)  $ as $v=u$
on the boundary $\mathbb{S}^{1}$. Because the boundary has geodesic curvature
$\geq c\geq0$, it is known that the smallest geodesic disc $D$ containing
$\Omega$ has radius at most $\cot^{-1}(c)$. Hence $L\left(  \gamma\right)
=L(\sigma)\leq2\pi/\sqrt{1+c^{2}}=L\left(  \partial D\right)  $. The equality
case follows directly from the argument.
\end{proof}
\end{theorem}

As a corollary we have the following theorem due to Toponogov.

\begin{corollary}
(Toponogov \cite{T}) Let $(M^{2},g)$ be a closed surface with Gaussian
curvature $K\geq1$. Then any simple closed geodesic in $M$ has length at most
$2\pi$. Moreover if there is one with length $2\pi$, then $M$ is isometric to
the standard sphere $\mathbb{S}^{2}$.
\end{corollary}

\begin{proof}
Suppose $\gamma$ is a simple close geodesic. We cut $M$ along $\gamma$ to
obtain two compact surfaces with the geodesic $\gamma$ as their common
boundary. The result follows from applying the previous theorem to either of
these two compact surfaces with boundary.
\end{proof}

\bigskip

Toponogov's original proof, as presented in Klingenberg \cite[page 297]{K}
uses his triangle comparison theorem. In applying the triangle comparison
theorem, which requires at least two minimizing geodesics, the difficulty is
to know how long a geodesic segment is minimizing without assuming an upper
bound for curvature. As the proof presented above, this difficulty is overcome
by using special features of two dimensional topology.

\section{\bigskip Conformal change of metrics}

Of course the conformal method is not very useful in higher dimensions. It may
still be of interest to record here what one can prove with it.

\begin{proposition}
Assume $\Omega\subset S_{+}^{n}$ is a smooth domain, $\widetilde{g}%
=u^{\frac{4}{n-2}}g_{S^{n}}$, $\left.  u\right\vert _{\partial\Omega}=1$,
$\widetilde{R}\geq R_{S^{n}}=n\left(  n-1\right)  $, then $u\geq1$,
$\widetilde{H}\leq H$. Moreover if equality holds somewhere, then
$\widetilde{g}=g_{S^{n}}$.
\end{proposition}

We have%
\[
\widetilde{R}=u^{-\frac{n+2}{n-2}}\left(  -\frac{4\left(  n-1\right)  }%
{n-2}\Delta u+n\left(  n-1\right)  u\right)  \geq n\left(  n-1\right)  ,
\]
in another way it is%
\[
-\Delta u+\frac{n\left(  n-2\right)  }{4}u\geq\frac{n\left(  n-2\right)  }%
{4}u^{\frac{n+2}{n-2}},\quad\left.  u\right\vert _{\partial\Omega}=1.
\]
Let $\overline{u}=\min\left\{  u,1\right\}  $, then $\overline{u}$ is
Lipschitz and it follows from Kato's inequality that in distribution sense%
\[
-\Delta\overline{u}\geq-\chi_{u<1}\Delta u\geq\frac{n\left(  n-2\right)  }%
{4}\left(  \overline{u}^{\frac{n+2}{n-2}}-\overline{u}\right)  ,\quad\left.
\overline{u}\right\vert _{\partial\Omega}=1.
\]
Indeed,
\[
\overline{u}=\frac{u+1}{2}-\frac{\left\vert u-1\right\vert }{2}.
\]
for $\varepsilon>0$, we let $f_{\varepsilon}\left(  t\right)  =\sqrt
{t^{2}+\varepsilon^{2}}$, then $\Delta\left(  f_{\varepsilon}\left(
u-1\right)  \right)  \geq f_{\varepsilon}^{\prime}\left(  u-1\right)  \Delta
u$, let $\varepsilon\rightarrow0^{+}$ it follows that%
\[
\Delta\left(  \left\vert u-1\right\vert \right)  \geq\operatorname{sgn}\left(
u-1\right)  \Delta u.
\]
Hence%
\[
-\Delta\overline{u}\geq-\frac{1}{2}\Delta u+\frac{1}{2}\operatorname{sgn}%
\left(  u-1\right)  \Delta u=-\chi_{u<1}\Delta u
\]
in distribution sense. Let $v\in C^{2}\left(  \overline{\Omega}\right)  $ such
that%
\[
\left\{
\begin{array}
[c]{l}%
-\Delta\overline{v}+\frac{n\left(  n-2\right)  }{4}\overline{v}=\frac{n\left(
n-2\right)  }{4}\overline{u}^{\frac{n+2}{n-2}},\\
\left.  \overline{v}\right\vert _{\partial\Omega}=1,
\end{array}
\right.
\]
then $\overline{v}\leq\overline{u}\leq1$ and hence $\frac{\partial\overline
{v}}{\partial\nu}\geq0$. Define%
\[
v\left(  x\right)  =\left\{
\begin{array}
[c]{l}%
\overline{v}\left(  x\right)  ,\text{ }x\in\Omega,\\
1,\text{ }x\notin\Omega.
\end{array}
\right.
\]
Then $v$ is Lipschitz and in the sense of distribution%
\[
-\Delta v+\frac{n\left(  n-2\right)  }{4}v\geq\frac{n\left(  n-2\right)  }%
{4}v^{\frac{n+2}{n-2}},\quad\left.  v\right\vert _{S^{n-1}}=1.
\]
Indeed for any nonnegative $\varphi\in C^{\infty}\left(  S_{+}^{n}\right)  $,
$\varphi=0$ near $S^{n-1}$, we have%
\begin{align*}
&  \int_{S_{+}^{n}}\left(  \nabla v\cdot\nabla\varphi+\frac{n\left(
n-2\right)  }{4}v\varphi\right)  d\mu\\
&  =\int_{\Omega}\left(  \nabla\overline{v}\cdot\nabla\varphi+\frac{n\left(
n-2\right)  }{4}\overline{v}\varphi\right)  d\mu+\int_{S_{+}^{n}%
\backslash\Omega}\frac{n\left(  n-2\right)  }{4}\varphi d\mu\\
&  \geq\int_{\partial\Omega}\frac{\partial\overline{v}}{\partial\nu}\varphi
dS+\int_{S_{+}^{n}}\frac{n\left(  n-2\right)  }{4}v^{\frac{n+2}{n-2}}d\mu\\
&  \geq\int_{S_{+}^{n}}\frac{n\left(  n-2\right)  }{4}v^{\frac{n+2}{n-2}}d\mu.
\end{align*}
Let $w\in C^{2}\left(  S_{+}^{n}\right)  $ satisfy%
\[
\left\{
\begin{array}
[c]{l}%
-\Delta w+\frac{n\left(  n-2\right)  }{4}w=\frac{n\left(  n-2\right)  }%
{4}v^{\frac{n+2}{n-2}}\\
\left.  w\right\vert _{S^{n-1}}=1
\end{array}
\right.  .
\]
It follows that $0\leq w\leq v$ and hence $-\Delta w+\frac{n\left(
n-2\right)  }{4}w\geq\frac{n\left(  n-2\right)  }{4}w^{\frac{n+2}{n-2}}$.
Using the conformal rigidity result in Hang and Wang \cite{HW} (Theorem 3.1 on p99)
 we see $w=1$. Hence
$v=1$ and $\overline{u}=1$. It follows that $u\geq1$. Note that%
\[
-\Delta u\geq\frac{n\left(  n-2\right)  }{4}\left(  u^{\frac{n+2}{n-2}%
}-u\right)  \geq0.
\]
Hence $u$ is superharmonic. It follows from strong maximum principle that
either $u\equiv1$ or $u>1$ in $\Omega$ and $\frac{\partial u}{\partial\nu}<0$.
The conclusion follows from%
\[
\widetilde{H}=\frac{2\left(  n-1\right)  }{n-2}\frac{\partial u}{\partial\nu
}+H.
\]

\section{\bigskip The proof of Theorem \ref{main}}

\bigskip We now present a proof of Theorem \ref{main} which works in any
dimension $n\geq2$. We first recall the following result due to Reilly.

\begin{theorem}
(Reilly \cite{R}) Let $\left(  M^{n},g\right)  $ be a compact Riemannian
manifold with nonempty boundary $\Sigma=\partial M$. Assume that
$\mathrm{Ric}\geq\left(  n-1\right)  g$ and the mean curvature of $\Sigma$ in
$M$ is nonnegative. Then the first (Dirichlet) eigenvalue $\lambda_{1}$ of
$-\Delta$ satisfies the inequality $\lambda_{1}\geq n$. Moreover $\lambda
_{1}=n$ iff $M$ is isometric to the standard hemisphere $\mathbb{S}_{+}%
^{n}\subset\mathbb{R}^{n+1}$.
\end{theorem}

Therefore to prove Theorem \ref{main}, it suffices to show $\lambda_{1}\left(
M\right)  =n$. If this were not the case, then $\lambda_{1}\left(  M\right)
>n$. Therefore for every $f\in C^{\infty}\left(  \Sigma\right)  $ there is a
unique $u\in C^{\infty}\left(  M\right)  $ solving%
\begin{equation}
\left\{
\begin{array}
[c]{ccc}%
-\Delta u=nu & \text{on} & M,\\
u=f & \text{on} & \Sigma.
\end{array}
\right.  \label{bv}%
\end{equation}
Define
\[
\phi=\left\vert \nabla u\right\vert ^{2}+u^{2}.
\]

\begin{lemma}
\label{sub}\bigskip\ $\phi$ is subharmonic, i.e. $\Delta\phi\geq0$.
\end{lemma}

\begin{proof}
Using the Bochner formula, the equation (\ref{bv}) and the assumption
\textrm{$Ric$}$\geq\left(  n-1\right)  g$,
\begin{align*}
\frac{1}{2}\Delta\phi &  =\left\vert D^{2}u\right\vert ^{2}+\left\langle
\nabla u,\nabla\Delta u\right\rangle +\mathrm{Ric}(\nabla u,\nabla
u)+\left\vert \nabla u\right\vert ^{2}+u\Delta u\\
&  \geq\left\vert D^{2}u\right\vert ^{2}-nu^{2}\\
&  \geq\frac{\left(  \Delta u\right)  ^{2}}{n}-nu^{2}\\
&  =0.
\end{align*}

\end{proof}

\bigskip

Denote $\chi=\frac{\partial u}{\partial\nu}$, the derivative on the boundary
in the direction of the outer unit normal field $\nu$. By the assumption of
Theorem \ref{main} there is an isometry $F:\left(  \Sigma,g|_{\Sigma}\right)
\rightarrow\mathbb{S}^{n-1}\subset\mathbb{R}^{n}$. In the following let
$f=\sum_{i=1}^{n}\alpha_{i}x_{i}\circ F$, where $x_{1},\cdots,x_{n}$ are the
standard coordinate functions on $\mathbb{S}^{n-1}$ and $\alpha=\left(
\alpha_{1},\cdots,\alpha_{n}\right)  \in\mathbb{S}^{n-1}$. We have%
\[
-\Delta_{\Sigma}f=\left(  n-1\right)  f,\text{ \ \ }\left\vert \nabla_{\Sigma
}f\right\vert ^{2}+f^{2}=1.
\]
Hence
\begin{equation}
\phi|_{\Sigma}=\left\vert \nabla_{\Sigma}f\right\vert ^{2}+\chi^{2}%
+f^{2}=1+\chi^{2}. \label{fb}%
\end{equation}
On the boundary $\Sigma$%
\[
-nf=\Delta u|_{\Sigma}=\Delta_{\Sigma}f+H\chi+D^{2}u\left(  \nu,\nu\right)
=-\left(  n-1\right)  f+H\chi+D^{2}u\left(  \nu,\nu\right)  ,
\]
whence
\begin{equation}
D^{2}u\left(  \nu,\nu\right)  +f=-H\chi. \label{d2un}%
\end{equation}

\begin{lemma}
\label{nf}\bigskip On $\Sigma$%
\[
\frac{1}{2}\frac{\partial\phi}{\partial\nu}=\left\langle \nabla_{\Sigma
}f,\nabla_{\Sigma}\chi\right\rangle -H\chi^{2}-\Pi\left(  \nabla_{\Sigma
}f,\nabla_{\Sigma}f\right)  .
\]

\end{lemma}

\begin{proof}
Indeed%
\begin{align*}
\frac{1}{2}\frac{\partial\phi}{\partial\nu}  &  =D^{2}u\left(  \nabla
u,\nu\right)  +f\chi\\
&  =D^{2}u\left(  \nabla_{\Sigma}u,\nu\right)  +\chi\left(  D^{2}u\left(
\nu,\nu\right)  +f\right) \\
&  =D^{2}u\left(  \nabla_{\Sigma}f,\nu\right)  -H\chi^{2},
\end{align*}
here we have used (\ref{d2un}) in the last step. On the other hand%
\begin{align*}
D^{2}u\left(  \nabla_{\Sigma}f,\nu\right)   &  =\left\langle \nabla
_{\nabla_{\Sigma}f}\nabla u,\nu\right\rangle \\
&  =\nabla_{\Sigma}f\left\langle \nabla u,\nu\right\rangle -\left\langle
\nabla u,\nabla_{\nabla_{\Sigma}f}\nu\right\rangle \\
&  =\left\langle \nabla_{\Sigma}f,\nabla_{\Sigma}\chi\right\rangle -\Pi\left(
\nabla_{\Sigma}f,\nabla_{\Sigma}f\right)  .
\end{align*}
The lemma follows.
\end{proof}

\begin{lemma}
\bigskip The function $\phi=\left\vert \nabla u\right\vert ^{2}+u^{2}$ is
constant and
\[
D^{2}u=-ug.
\]
Moreover $\chi=\frac{\partial u}{\partial\nu}$ is also constant and
$\Pi\left(  \nabla_{\Sigma}f,\nabla_{\Sigma}f\right)  \equiv0$.
\end{lemma}

\begin{proof}
Since $\phi$ is subharmonic, by the maximum principle $\phi$ achieves its
maximum on $\Sigma$, say at $p\in\Sigma$. Obviously we have
\[
\nabla_{\Sigma}\phi\left(  p\right)  =0,\text{ \ }\frac{\partial\phi}%
{\partial\nu}\left(  p\right)  \geq0.
\]
If $\frac{\partial\phi}{\partial\nu}\left(  p\right)  =0$, then $\phi$ must be
constant by the strong maximum principle and Hopf lemma (see \cite[page
34-35]{GT}). Then the proof of Lemma \ref{sub} implies $D^{2}u=-ug$. By
(\ref{fb}) $\chi$ is constant. It then follows from Lemma \ref{nf} that
$\Pi\left(  \nabla_{\Sigma}f,\nabla_{\Sigma}f\right)  \equiv0$.

Suppose $\frac{\partial\phi}{\partial\nu}\left(  p\right)  >0$. Then
$\chi\left(  p\right)  \neq0$, for otherwise it follows from (\ref{fb}) that
$\chi\equiv0$ and hence $\frac{\partial\phi}{\partial\nu}\left(  p\right)
\leq0$ by Lemma \ref{nf}, a contradiction. From (\ref{fb}) we conclude
$\nabla_{\Sigma}\chi\left(  p\right)  =0$. By Lemma \ref{nf}
\[
\frac{1}{2}\frac{\partial\phi}{\partial\nu}\left(  p\right)  =\left\langle
\nabla_{\Sigma}f,\nabla_{\Sigma}\chi\right\rangle \left(  p\right)  -H\chi
^{2}-\Pi\left(  \nabla_{\Sigma}f,\nabla_{\Sigma}f\right)  \leq0,
\]
here we have used the assumption that $\Sigma$ is convex, i.e. $\Pi\geq0$.
This contradicts with $\frac{\partial\phi}{\partial\nu}\left(  p\right)  >0$ again.
\end{proof}

\bigskip

Recall $f$ depends on a unit vector $\alpha\in\mathbb{S}^{n-1}$. To indicate
the dependence on $\alpha$ we will add subscript $\alpha$ to all the
quantities. Since $\Pi\left(  \nabla_{\Sigma}f_{\alpha},\nabla_{\Sigma
}f_{\alpha}\right)  \equiv0$ on $\Sigma$ for any $\alpha\in\mathbb{S}^{n-1}$
and $\left\{  \nabla_{\Sigma}f_{\alpha}:\alpha\in\mathbb{S}^{n-1}\right\}  $
span the tangent bundle $T\Sigma$ we conclude that $\Sigma$ is totally
geodesic, i.e. $\Pi=0$.

We now claim that we can choose $\alpha$ such that $\chi_{\alpha}\equiv0$.
Indeed, $\alpha\rightarrow\chi_{\alpha}$ is a continuous function on
$\mathbb{S}^{n-1}$. Clearly $u_{-\alpha}=-u_{\alpha}$ and hence $\chi
_{-\alpha}=-\chi_{\alpha}$. Therefore by the intermediate value theorem there
exists some $\beta\in\mathbb{S}^{n-1}$ such that $\chi_{\beta}\equiv0$. With
this particular choice $f=f_{\beta},u=u_{\beta}$ we have%
\[
\left\{
\begin{array}
[c]{c}%
D^{2}u=-ug,\\
\frac{\partial u}{\partial\nu}\equiv0.
\end{array}
\right.
\]
There is $q\in\Sigma$ such that $f\left(  q\right)  =\max f=1$. Then
$\nabla_{\Sigma}f\left(  q\right)  =0$ and hence $\nabla u\left(  q\right)
=0$ as $\frac{\partial u}{\partial\nu}\left(  q\right)  =0$. For $X\in T_{q}M$
such that $\left\langle X,\nu\left(  q\right)  \right\rangle \leq0$ let
$\gamma_{X}$ be the geodesic with $\overset{\cdot}{\gamma}_{X}\left(
0\right)  =X$. Note that $\gamma_{X}$ lies in $\Sigma$ if $X$ is tangential to
$\Sigma$ since $\Sigma$ is totally geodesic. The function $U\left(  t\right)
=u\circ\gamma_{X}\left(  t\right)  $ then satisfies the following%
\[
\left\{
\begin{array}
[c]{c}%
\overset{\cdot\cdot}{U}\left(  t\right)  =-U,\\
U\left(  0\right)  =1,\\
\overset{\cdot}{U}\left(  0\right)  =0.
\end{array}
\right.
\]
Hence $U\left(  t\right)  =\cos t$. Because $\Sigma$ is totally geodesic,
every point may be connected to $q$ by a minimizing geodesic. Using the
geodesic polar coordinates $\left(  r,\xi\right)  \in\mathbb{R}^{+}%
\times\mathbb{S}_{+}^{n-1}$at $q$ we can write%
\[
g=dr^{2}+h_{r}%
\]
where $r$ is the distance function to $q$ and $h_{r}$ is $r$-family of metrics
on $\mathbb{S}_{+}^{n-1}$ with
\[
\lim_{r\rightarrow0}r^{-2}h_{r}=h_{0},
\]
here $h_{0}$ is the standard metric on $\mathbb{S}_{+}^{n-1}$. Then $u=\cos r
$. The equation $D^{2}u=-ug$ implies%
\[
\frac{\partial h_{r}}{\partial r}=2\frac{\cos r}{\sin r}h_{r}%
\]
which can be solved to give $h_{r}=\sin^{2}rh_{0}$. It follows that $\left(
M,g\right)  $ is isometric to $\mathbb{S}_{+}^{n}$. This implies $\lambda
_{1}\left(  M\right)  =n$ and contradicts with the assumption $\lambda
_{1}\left(  M\right)  >n$. Theorem \ref{main} follows.

\section{\bigskip Reilly's theorem and rigidity for certain nonsmooth metrics}

To prove Theorem \ref{general}, it is natural to try the same argument of
Section 4. Namely, we take $v$ to be a linear function on $\mathbb{S}^{n}$ and
then solve
\[
\left\{
\begin{array}
[c]{ccc}%
-\Delta u=nu & \text{on} & M,\\
u=v\circ\iota & \text{on} & \Sigma.
\end{array}
\right.
\]

As before, $\phi=\left\vert \nabla u\right\vert ^{2}+u^{2}$ is subharmonic.
But when we apply the strong maximum principle to $\phi$, we inevitably have
to compare $\frac{\partial u}{\partial\nu}$ with the corresponding quantity
$\frac{\partial v}{\partial\nu}\circ\iota$. We have no idea how such a
comparison could be established. Instead, we have to take a different route.

First, we generalize Reilly's theorem to the situation where the metric $g$ is
only Lipschitz along a hypersurface. To be precise, let $M$ be a smooth
compact Riemannian manifold with $\partial M=\Sigma$, $\mathrm{Ric}\geq\left(
n-1\right)  $. Let $N$ be another smooth compact Riemannian manifold with
$\partial N=\Sigma\cup\Sigma_{1}$, $\Sigma$ and $\Sigma_{1}$ being disjoint
components, and $\mathrm{Ric}\geq\left(  n-1\right)  $. Assume $\left.
g_{M}\right\vert _{\Sigma}=\left.  g_{N}\right\vert _{\Sigma}$. Now we glue
$M$ and $N$ along $\Sigma$ to get a smooth manifold $P$ with boundary
$\Sigma_{1}$. However the metric on $P$ is only Lipschitz along $\Sigma$. Let
$\nu$ be the outer normal direction of $M$ along $\Sigma$. We have two shape
$A_{M}\left(  X\right)  =\nabla_{X}^{M}\nu,A_{N}\left(  X\right)  =\nabla
_{X}^{N}\nu$ for $X\in T\Sigma$. For $X\in T\Sigma_{1}$, $A\left(  X\right)
=\nabla_{X}\nu$, here $\nu$ is the outer normal direction for $N$ along
$\Sigma_{1}$, and $H=\operatorname*{tr}A$ is the mean curvature.

\begin{theorem}
\label{thm5.2}Assume $A_{M}\geq A_{N}$ and $H\geq0$, then $\lambda_{1}\left(
P\right)  \geq n$. If $\lambda_{1}\left(  P\right)  =n$, then $\left(
P,g\right)  $ is smooth and $\left(  P,g\right)  $ is isometric to $\left(
\mathbb{S}_{+}^{n},g_{\mathbb{S}^{n}}\right)  $.
\end{theorem}

\bigskip

\begin{proof}
Let $u\in H_{0}^{1}\left(  P\right)  $ be the first eigenfunction, then
$u\geq0$ and $-\Delta u=\lambda u$, $\lambda>0$. It follows from elliptic
regularity theory that $\left.  u\right\vert _{M}\in C^{\infty}\left(
M\right)  $, $\left.  u\right\vert _{N}\in C^{\infty}\left(  N\right)  $ and%
\[
\frac{\partial\left.  u\right\vert _{M}}{\partial\nu}=\frac{\partial\left.
u\right\vert _{N}}{\partial\nu}\quad\text{on }\Sigma\text{.}%
\]
In particular $u\in C^{1,1}\left(  P\right)  $. Applying Reilly's formula on
$M $, we get%
\begin{align}
&  \frac{1}{2}\int_{M}\left(  \left(  \Delta u\right)  ^{2}-\left\vert
D^{2}u\right\vert ^{2}\right)  d\mu\nonumber\\
&  =\frac{1}{2}\int_{M}\mathrm{Ric}\left(  \nabla u,\nabla u\right)  d\mu
+\int_{\Sigma}\Delta_{\Sigma}u\cdot\frac{\partial u}{\partial\nu}dS+\frac
{1}{2}\int_{\Sigma}H_{M}\left(  \frac{\partial u}{\partial\nu}\right)
^{2}dS\nonumber\\
&  +\frac{1}{2}\int_{\Sigma}\left\langle A_{M}\left(  \nabla_{\Sigma}u\right)
,\nabla_{\Sigma}u\right\rangle dS,\nonumber
\end{align}
here $H_{M}=\operatorname*{tr}A_{M}$. Applying the same formula on $N$ yields%
\begin{align}
&  \frac{1}{2}\int_{N}\left(  \left(  \Delta u\right)  ^{2}-\left\vert
D^{2}u\right\vert ^{2}\right)  d\mu\nonumber\\
&  =\frac{1}{2}\int_{N}\mathrm{Ric}\left(  \nabla u,\nabla u\right)  d\mu
-\int_{\Sigma}\Delta_{\Sigma}u\cdot\frac{\partial u}{\partial\nu}dS-\frac
{1}{2}\int_{\Sigma}H_{N}\left(  \frac{\partial u}{\partial\nu}\right)
^{2}dS\nonumber\\
&  -\frac{1}{2}\int_{\Sigma}\left\langle A_{N}\left(  \nabla_{\Sigma}u\right)
,\nabla_{\Sigma}u\right\rangle dS+\frac{1}{2}\int_{\Sigma_{1}}H\left(
\frac{\partial u}{\partial\nu}\right)  ^{2}dS.\nonumber
\end{align}
Summing up we get%
\begin{align*}
&  \frac{1}{2}\int_{P}\left(  \left(  \Delta u\right)  ^{2}-\left\vert
D^{2}u\right\vert ^{2}\right)  d\mu\\
&  =\frac{1}{2}\int_{P}\mathrm{Ric}\left(  \nabla u,\nabla u\right)
d\mu+\frac{1}{2}\int_{\Sigma}\left(  H_{M}-H_{0}\right)  \left(
\frac{\partial u}{\partial\nu}\right)  ^{2}dS\\
&  +\frac{1}{2}\int_{\Sigma}\left\langle \left(  A_{M}-A_{N}\right)  \left(
\nabla_{\Sigma}u\right)  ,\nabla_{\Sigma}u\right\rangle dS+\frac{1}{2}%
\int_{\Sigma_{1}}H\left(  \frac{\partial u}{\partial\nu}\right)  ^{2}dS.
\end{align*}
Note that%
\[
\left\vert D^{2}u\right\vert ^{2}=\left\vert D^{2}u-\frac{\Delta u}%
{n}g\right\vert ^{2}+\frac{\left(  \Delta u\right)  ^{2}}{n}=\left\vert
D^{2}u-\frac{\lambda u}{n}g\right\vert ^{2}+\frac{\lambda^{2}u^{2}}{n}.
\]
Hence%
\begin{align*}
&  \frac{n-1}{n}\lambda^{2}\int_{M}u^{2}d\mu\\
&  \geq\left(  n-1\right)  \int_{M}\left\vert \nabla u\right\vert ^{2}%
d\mu+\int_{M}\left\vert D^{2}u-\frac{\lambda u}{n}g\right\vert ^{2}d\mu
+\int_{\Sigma_{1}}H\left(  \frac{\partial u}{\partial\nu}\right)  ^{2}dS\\
&  \geq\left(  n-1\right)  \lambda\int_{M}u^{2}d\mu+\int_{M}\left\vert
D^{2}u-\frac{\lambda u}{n}g\right\vert ^{2}d\mu+\int_{\Sigma_{1}}H\left(
\frac{\partial u}{\partial\nu}\right)  ^{2}dS\\
&  \geq\left(  n-1\right)  \lambda\int_{M}u^{2}d\mu.
\end{align*}
Hence $\lambda\geq n$.

If $\lambda=n$, then $D^{2}u=-ug$ on both $M$ and $N$ and $H\left(
\frac{\partial u}{\partial\nu}\right)  ^{2}=0$ on $\Sigma_{1}$. Since $u>0$ in
$P\backslash\Sigma_{1}$, it follows from strong maximum principle that
$\frac{\partial u}{\partial\nu}<0$ on $\Sigma_{1}$ and hence $H=0$ on
$\Sigma_{1}$. We aim to show $\left(  P,g\right)  $ is in fact isometric to
$\left(  \mathbb{S}_{+}^{n},g_{\mathbb{S}_{+}^{n}}\right)  $. The key is to
prove that $g\in C^{\infty}$.

To continue we build some coordinates along $\Sigma$. Note for $r\geq0$, we
have a map $\Sigma\times\left[  0,\varepsilon\right)  \rightarrow M:\left(
p,r\right)  \mapsto\exp_{p}\left(  -r\nu\left(  p\right)  \right)  $ which is
an smooth embedding when $\varepsilon$ is small. If we choose a coordinate
locally on $\Sigma$, namely $\theta_{1},\cdots,\theta_{n-1}$, hence we have a
coordinate $r,\theta_{1},\cdots,\theta_{n-1}$ near $\Sigma$. Similarly using
the map $\Sigma\times\left(  -\varepsilon,0\right]  \rightarrow N:\left(
p,r\right)  \mapsto\exp_{p}\left(  -r\nu\left(  p\right)  \right)  $ we have
coordinate $r,\theta_{1},\cdots,\theta_{n-1}$ near $\Sigma$ on $P$. Note that%
\[
g=dr\otimes dr+b_{ij}\left(  r,\theta\right)  d\theta_{i}\otimes d\theta_{j}.
\]
$b_{ij}\left(  r,\theta\right)  $ is Lipschitz. We will write $u_{0}%
=\partial_{r}u,u_{i}=\partial_{i}u,u_{00}=D^{2}u\left(  \partial_{r}%
,\partial_{r}\right)  ,u_{ij}=D^{2}u\left(  \partial_{i},\partial_{j}\right)
$ etc.

From $D^{2}u=-ug$ it is easy to see $\nabla\left(  \left\vert \nabla
u\right\vert ^{2}+u^{2}\right)  =0$ on both $M$ and $N$. Since $u\in C^{1}$ we
conclude $\left\vert \nabla u\right\vert ^{2}+u^{2}=\operatorname*{const}$. By
scaling we may assume $\left\vert \nabla u\right\vert ^{2}+u^{2}=1$. We first
observe that $u\in C^{\infty}\left(  P\right)  $. Indeed, to see this we only
need to show $\partial_{r}^{m}u\left(  0^{+},\theta\right)  =\partial_{r}%
^{m}u\left(  0^{-},\theta\right)  $ for all $m$. But since $-u=u_{00}%
=\partial_{r}^{2}u$, and $u\left(  0,\theta\right)  =a\left(  \theta\right)
$, $\partial_{r}u\left(  0,\theta\right)  =b\left(  \theta\right)  $, we see
$u\left(  r,\theta\right)  =a\left(  \theta\right)  \cos r+b\left(
\theta\right)  \sin r$ for both positive and negative $r$, hence $\partial
_{r}^{m}u\left(  0^{+},\theta\right)  =\partial_{r}^{m}u\left(  0^{-}%
,\theta\right)  $ for all $m$.

Next we observe that if $u\left(  p\right)  =1$, then at $p$, $D^{2}u=-g$, it
follows that $u<1$ for other points near $p$. Hence the set $\left\{
u=1\right\}  $ is discrete. On $\left\{  u\neq1\right\}  $, $\left\vert \nabla
u\right\vert =\sqrt{1-u^{2}}$ is smooth too.

Next we claim $\nabla u$ is a smooth vector field, though apriori it seems
only belongs to Lipschitz. Indeed we have%
\[
\nabla u=u_{0}\partial_{r}+b^{ij}u_{j}\partial_{i}.
\]
We need to show $\partial_{r}^{m}\left(  b^{ij}u_{j}\right)  \left(
0^{+},\theta\right)  =\partial_{r}^{m}\left(  b^{ij}u_{j}\right)  \left(
0^{-},\theta\right)  $.

Given $p\in\Sigma$, if $\partial_{r}u\left(  p\right)  \neq0$, then it is not
zero near $p$. Note that%
\[
-ub_{ij}=u_{ij}=\partial_{ij}u+\frac{1}{2}\partial_{r}b_{ij}u_{0}-\Gamma
_{ij}^{k}u_{k}.
\]
Restricting to $r=0$ on both sides, we see $\partial_{r}b_{ij}\left(
0^{+},\theta\right)  =\partial_{r}b_{ij}\left(  0^{-},\theta\right)  $ for
$\theta$ near $\theta\left(  p\right)  $. Using%
\[
\partial_{r}b_{ij}=\frac{2}{u_{0}}\left(  \Gamma_{ij}^{k}u_{k}-ub_{ij}%
-\partial_{ij}u\right)  ,
\]
and the fact $u\in C^{\infty}$ we see $\partial_{r}^{2}b_{ij}\left(
0^{+},\theta\right)  =\partial_{r}^{2}b_{ij}\left(  0^{-},\theta\right)  $. By
induction we see $\partial_{r}^{m}b_{ij}\left(  0^{+},\theta\right)
=\partial_{r}^{m}b_{ij}\left(  0^{-},\theta\right)  $ for all $m$. Hence $g$
is smooth near $p$ and
\[
\partial_{r}^{m}\left(  b^{ij}u_{j}\right)  \left(  0^{+},\theta\right)
=\partial_{r}^{m}\left(  b^{ij}u_{j}\right)  \left(  0^{-},\theta\right)
\]
for all $m$.

Now assume $\partial_{r}u\left(  p\right)  =0$. If there exists a sequence
$p_{i}\in\Sigma$ with $\partial_{r}u\left(  p_{i}\right)  \neq0$ such that
$p_{i}\rightarrow p$, then by taking limit of what we have at $p_{i}$ we
obtain
\[
\partial_{r}^{m}\left(  b^{ij}u_{j}\right)  \left(  0^{+},\theta\left(
p\right)  \right)  =\partial_{r}^{m}\left(  b^{ij}u_{j}\right)  \left(
0^{-},\theta\left(  p\right)  \right)
\]
for all $m$.

If $\partial_{r}u\left(  q\right)  =0$ for $q\in\Sigma$ near $p$. Denote
$u\left(  0,\theta\right)  =f\left(  \theta\right)  $. Using $\partial_{r}%
^{2}=-u$ on both sides, we see $u\left(  r,\theta\right)  =f\left(
\theta\right)  \cos r$. It follows from $u_{0i}=0$ that%
\[
\partial_{r}b_{ik}b^{kj}f_{j}=-2f_{i}\tan r.
\]
Hence%
\[
\partial_{r}\left(  b^{ij}f_{j}\right)  =-b^{ik}\partial_{r}b_{kl}b^{lj}%
f_{j}=-2b^{ik}f_{k}\tan r.
\]
Note that this is true for both positive and negative $r$. Hence%
\[
\left(  b^{ij}f_{j}\right)  \left(  r,\theta\right)  =\left(  b^{ij}%
f_{j}\right)  \left(  0,\theta\right)  \cos^{2}r.
\]
Hence $\left(  b^{ij}u_{j}\right)  \left(  r,\theta\right)  =\left(
b^{ij}f_{j}\right)  \left(  0,\theta\right)  \cos^{3}r$ and it follows that
$\partial_{r}^{m}\left(  b^{ij}u_{j}\right)  \left(  0^{+},\theta\right)
=\partial_{r}^{m}\left(  b^{ij}u_{j}\right)  \left(  0^{-},\theta\right)  $
for all $m$.

In any case we have proved that $\nabla u\in C^{\infty}\left(  P\right)  $.
Let $X=\frac{\nabla u}{\left\vert \nabla u\right\vert }$, then $X$ is
$C^{\infty}$ on $\left\{  u\neq1\right\}  $. It follows from $D^{2}u=-ug$ on
both $M$ and $N$ that $\nabla_{X}X=0$ on both $M$ and $N$. For $p\in\Sigma$,
let $F\left(  p,t\right)  $ be given by $\partial_{t}F\left(  p,t\right)
=X\left(  F\left(  p,t\right)  \right)  $ and $F\left(  p,0\right)  =p$. Then
$F$ is $C^{\infty}$ as long as it is defined, besides when $F\left(
p,t\right)  \in M$ on a time interval, then it is a unit speed geodesic.

Let $\phi\left(  t\right)  =u\left(  F\left(  p,t\right)  \right)  $, then
$\phi\left(  0\right)  =0$, $\phi^{\prime}\left(  t\right)  =\left\vert \nabla
u\left(  F\left(  p,t\right)  \right)  \right\vert $, hence $\phi^{\prime
}\left(  0\right)  =1$. Since $\phi^{2}+\phi^{\prime2}=1$ and $\phi^{\prime
}\neq0$, we see $\phi^{\prime\prime}+\phi=0$. Hence $u\left(  F\left(
p,t\right)  \right)  =\sin t$. It follows that $F:\Sigma_{1}\times\left[
0,\frac{\pi}{2}\right)  \rightarrow\left\{  u\neq1\right\}  $ is a
diffeomorphism. If we choose a local coordinate $\theta_{1},\cdots
,\theta_{n-1}$ on $\Sigma_{1}$, then note that we have a coordinate
$t,\theta_{1},\cdots,\theta_{n-1}$ locally on $P$ with $\partial_{t}=X$, hence
$\left\vert \partial_{t}\right\vert =1$ and $\left\langle \partial
_{t},\partial_{i}\right\rangle =0$. It follows that%
\[
g=dt\otimes dt+b_{ij}\left(  t,\theta\right)  d\theta_{i}\otimes d\theta_{j},
\]
here $b_{ij}$ is locally Lipschitz in $\left(  t,\theta\right)  $. If $\left(
t_{0},\theta_{0}\right)  \notin\Sigma$, then we know near $\left(
t_{0},\theta_{0}\right)  $, $b_{ij}$ is smooth with $u=\sin t$. Hence%
\begin{align*}
D^{2}u  &  =-\sin t\cdot dt\otimes dt+\frac{1}{2}\cot s\cdot\partial_{t}%
b_{ij}d\theta_{i}\otimes d\theta_{j}\\
&  =-\sin t\cdot dt\otimes dt-\sin t\cdot b_{ij}\left(  t,\theta\right)
d\theta_{i}\otimes d\theta_{j}.
\end{align*}
Hence $\partial_{t}b_{ij}\left(  t,\theta\right)  =-2\tan t\cdot b_{ij}\left(
t,\theta\right)  $. For a.e. $\theta$, $\left(  t,\theta\right)  \in\Sigma$
for only a set discrete $t$'s, hence we have $b_{ij}\left(  t,\theta\right)
=b_{ij}\left(  0,\theta\right)  \cos^{2}t$. By continuity we know this is true
for all $\theta$. Therefore $g$ is smooth on $\left\{  u\neq1\right\}  $.
Since the set $\left\{  u=1\right\}  $ is finite, we see $g$ is smooth
everywhere (because by taking a limit we see there is no jump in any order of
derivatives along $\Sigma$). Hence $\left(  P,g\right)  $ must be isometric to
the standard upper half sphere.
\end{proof}

\bigskip

We can now prove Theorem \ref{general}. Let $N=\mathbb{S}_{+}^{n}%
\backslash\Omega$ and let $P$ be the smooth manifold obtained by gluing $M$
and $N$ along $\Sigma$ via the embedding $\iota:\left(  \Sigma,g_{\Sigma
}\right)  \rightarrow\partial\overline{\Omega}$. We have a Riemannian metric
on $P$ which is merely Lipschitz along $\Sigma$. \ Notice that $P$ is
spherical near its boundary $\mathbb{S}^{n-1}$. By Theorem \ref{thm5.2},
$\lambda_{1}\left(  P\right)  \geq n$. If $\lambda_{1}\left(  P\right)  =n$
then we know $\left(  P,g\right)  $ is smooth and is isometric to
$\mathbb{S}_{+}^{n}$. Assume $\lambda_{1}\left(  P\right)  >n$, then we may
find $u\in H^{1}\left(  P\right)  $ s.t.%
\[
\left\{
\begin{array}
[c]{ccc}%
-\Delta u=nu & \text{on} & P,\\
u=f & \text{on} & \mathbb{S}^{n-1},
\end{array}
\right.
\]
here $f$ is a linear function on $\mathbb{S}^{n-1}$. By elliptic regularity
$\left.  u\right\vert _{M}\in C^{\infty}\left(  M\right)  $, $\left.
u\right\vert _{N}\in C^{\infty}\left(  N\right)  $ and $u\in C^{1}\left(
P\right)  $. Let $\phi=\left\vert \nabla u\right\vert ^{2}+u^{2}$. We know
that $\phi$ is subharmonic in both $M$ and $N$. Moreover on $\Sigma$, let
$\chi=\frac{\partial u}{\partial\nu}$, we have%
\begin{align*}
\frac{1}{2}\frac{\partial\left.  \phi\right\vert _{M}}{\partial\nu}  &
=\left\langle \nabla_{\Sigma}u,\nabla_{\Sigma}\chi\right\rangle -\left\langle
A_{M}\left(  \nabla_{\Sigma}u\right)  ,\nabla_{\Sigma}u\right\rangle
+\chi\left[  -\Delta_{\Sigma}u-\left(  n-1\right)  u\right]  -H_{M}\chi^{2},\\
\frac{1}{2}\frac{\partial\left.  \phi\right\vert _{N}}{\partial\nu}  &
=\left\langle \nabla_{\Sigma}u,\nabla_{\Sigma}\chi\right\rangle -\left\langle
A_{0}\left(  \nabla_{\Sigma}u\right)  ,\nabla_{\Sigma}u\right\rangle
+\chi\left[  -\Delta_{\Sigma}u-\left(  n-1\right)  u\right]  -H_{0}\chi^{2}.
\end{align*}
Hence $\frac{\partial\left.  \phi\right\vert _{M}}{\partial\nu}\leq
\frac{\partial\left.  \phi\right\vert _{N}}{\partial\nu}$. Then for
$\varphi\in C_{c}^{\infty}\left(  P\right)  $, $\varphi\geq0$, we have%
\begin{align*}
&  \int_{P}\left\langle \nabla\phi,\nabla\varphi\right\rangle d\mu\\
&  =\int_{M}\left\langle \nabla\phi,\nabla\varphi\right\rangle d\mu+\int
_{N}\left\langle \nabla\phi,\nabla\varphi\right\rangle d\mu\\
&  =-\int_{M}\varphi\Delta\phi d\mu+\int_{\Sigma}\frac{\partial\left.
\phi\right\vert _{M}}{\partial\nu}\varphi dS-\int_{N}\varphi\Delta\phi
d\mu-\int_{\Sigma}\frac{\partial\left.  \phi\right\vert _{N}}{\partial\nu
}\varphi dS\\
&  \leq0.
\end{align*}
That is $\phi$ is subharmonic on $P$ in the distribution sense. Hence $\phi$
achieve a maximum on the boundary $\mathbb{S}^{n-1}$. At this maximum point,
we have $\frac{\partial\phi}{\partial\nu}=0$ by the same argument in Section
4. Hence $\phi$ must be equal to constant by the strong maximum principle and
$D^{2}u=-ug$ on both $M$ and $N$, $\left.  u\right\vert _{N}=$the linear
function. We may assume $\frac{\partial u}{\partial\nu}=0$ on $\mathbb{S}%
^{n-1}$. Hence $\left\vert \nabla u\right\vert ^{2}+u^{2}=1$ on $P$. It
follows from this and $D^{2}u=-ug$ on both $M$ and $N$ that $\left\{
u=\pm1\right\}  $ is finite.

The same argument as before shows that $u$ and $\nabla u$ belong to
$C^{\infty}\left(  P\right)  $. Since $\left\vert \nabla u\right\vert
=\sqrt{1-u^{2}}$ we see it is smooth on $\left\{  u\neq\pm1\right\}  $. Let
$X=\frac{\nabla u}{\left\vert \nabla u\right\vert }$, then it generates a
smooth flow on $\left\{  u\neq\pm1\right\}  $, $F\left(  p,t\right)  $ with
$\partial_{t}F\left(  p,t\right)  =X\left(  F\left(  p,t\right)  \right)  $,
$F\left(  p,0\right)  =p$. Note that here we have used the fact $\nabla u$ is
tangent to $\mathbb{S}^{n-1}$ on $\mathbb{S}^{n-1}$.

If $p\in P$, $u\left(  p\right)  \neq\pm1$, let $\phi\left(  t\right)
=u\left(  F\left(  p,t\right)  \right)  $, then $\phi^{\prime}\left(
t\right)  =\left\vert \nabla u\left(  F\left(  p,t\right)  \right)
\right\vert >0$. Hence $\phi^{\prime2}+\phi^{2}=1$. After differentiation we
get $\phi^{\prime\prime}+\phi=0$, hence $\phi\left(  t\right)  =\cos\left(
t+b\right)  $ for some $-\pi<b<0$. It exists on $\left(  -\pi-b,-b\right)  $.
Note that $\left\vert \partial_{t}F\left(  p,t\right)  \right\vert =1$, we see
$F\left(  p,t\right)  \rightarrow p_{+}$ as $t\rightarrow-b$ from the left and
$F\left(  p,t\right)  \rightarrow p_{-}$ as $t\rightarrow-\pi-b$ from the
right. In particular $u\left(  p_{+}\right)  =1$ and $u\left(  p_{-}\right)
=-1$. It follows that each orbit must have length $\pi$ and connecting some
points with value $-1$ to another point with value $1$.

For every $q$ with $u\left(  q\right)  =1$, we let%
\[
U_{q}=\cup\left\{  \text{all orbits ending at }q\right\}  .
\]
Then it is clear that $U_{q}$ is open and $\cup_{u\left(  q\right)  =1}%
U_{q}=P\backslash\left\{  u=\pm1\right\}  $. It follows from connectivity of
$P\backslash\left\{  u=\pm1\right\}  $ that there is only one $q$ with
$u\left(  q\right)  =1$. Similarly there is only one $q$ with $u\left(
q\right)  =-1$. Let $p_{+}\in\mathbb{S}^{n-1}$ with $u\left(  p_{+}\right)
=1$, $p_{-}\in\mathbb{S}^{n-1}$ with $u\left(  p_{-}\right)  =-1$, then every
orbit must start from $p_{-}$ and end at $p_{+}$. Next calculation shows that
in the interior of $M$ and $N$, $D_{X}X=0$, hence the orbits in the interior
are simply the unit speed geodesics. Let $r$ be the distance to $p_{-}$, then
near $p_{-}$ the metric%
\[
g=dr\otimes dr+\sin^{2}rb_{ij}\left(  \theta\right)  d\theta_{i}\otimes
d\theta_{j}.
\]
Here $\theta_{1},\cdots,\theta_{n-1}$ are local coordinates on $\mathbb{S}%
_{+}^{n-1}$ (viewed as in the tangent space of $\mathbb{S}_{+}^{n}$ at $p_{-}
$). Moreover $u\left(  r,\theta\right)  =-\cos r$. We have a diffeomorphism%
\[
\mathbb{S}_{+}^{n-1}\times\left(  0,\pi\right)  \rightarrow P\backslash
\left\{  p_{+},p_{-}\right\}  :\left(  \xi,t\right)  \mapsto F\left(
\exp_{p_{-}}\left(  \varepsilon\xi\right)  ,t-\varepsilon\right)
\]
for some $\varepsilon>0$ small. Hence we have a coordinate $t,\theta
_{1},\cdots,\theta_{n-1}$ on $P\backslash\left\{  p_{+},p_{-}\right\}  $.
Under this coordinate%
\[
g=dt\otimes dt+b_{ij}\left(  t,\theta\right)  d\theta_{i}\otimes d\theta_{j}.
\]
It follows from previous calculation that $u\left(  t,\theta\right)  =-\cos
t$. Note that if $\left(  t,\theta\right)  \notin\Sigma$, then it follows from
$D^{2}u=-ug$ that%
\[
\partial_{t}b_{ij}\left(  t,\theta\right)  =2\cot tb_{ij}\left(
t,\theta\right)  .
\]
For $a.e.$ $\theta$, we know $\left(  t,\theta\right)  \notin\Sigma$ except
finite many $t^{\prime}s$. Hence using $b_{ij}\left(  t,\theta\right)
=\sin^{2}tb_{ij}\left(  \theta\right)  $ for $t$ small we see $b_{ij}\left(
t,\theta\right)  =\sin^{2}tb_{ij}\left(  \theta\right)  $ for all $t$. By
continuity argument we see it is true everywhere. Hence $g$ is smooth. The
theorem follows.

\end{document}